\documentclass[a4paper]{amsart} 

\newcommand{\R}{\mathbb R}
\newcommand{\C}{\mathbb C}
\newcommand{\N}{\mathbb N}
\newcommand{\Z}{\mathbb Z}

\newcommand{\eps}{\varepsilon}
\renewcommand{\epsilon}{\varepsilon}

\newcommand{\ran}{\operatorname{ran}}

\AtBeginDocument{

}

\usepackage{etex}
\usepackage[english]{babel}
\usepackage{tikz}
\usetikzlibrary{arrows,decorations.pathmorphing,backgrounds,positioning,fit,petri,patterns}
\usepackage{pgfplots}
\usepackage[utf8]{inputenc}
\usepackage{graphicx}
\usepackage[colorlinks=true,linkcolor=red,citecolor=blue]{hyperref}
\usepackage{amsfonts}
\usepackage{enumerate}
\usepackage{booktabs}
\usepackage{array} 
\usepackage{paralist} 
\usepackage{subfig} 
\usepackage{mathtools}
\usepackage{bbm}
\usepackage{tabu}
\usepackage{amsthm}
\usepackage{empheq}
\usepackage{amsopn}
\usepackage{dsfont}
\usepackage{wrapfig}
\usepackage[font=small]{caption}
\usepackage{units}
\usepackage[symbol]{footmisc}
\usepackage{todonotes}
\usepackage{amssymb}
\usepackage{pdfpages}
\usepackage{setspace}
\usepackage{float}
\usepackage{scrextend}

\makeatletter

\def\env@cases#1{%
  \let\@ifnextchar\new@ifnextchar
  \left\lbrace\def\arraystretch{1.2}%
  \array{@{}#1@{\quad}l@{}}}
\makeatother

\usepackage{kvsetkeys}
\usepackage{etexcmds}

\makeatletter
\newcount\arg@count
\newcommand{\arg@parser}[1]{%
  \advance\arg@count\@ne
  \expandafter\let\csname arg\romannumeral\arg@count\endcsname\comma@entry
}
\newcommand\res[1]{
  \arg@count=\z@
  \comma@parse{ \lambda,A }\arg@parser 
  \arg@count=\z@
  \comma@parse{#1}\arg@parser
  \ifnum\arg@count>2 %
    \@latex@error{Too many arguments}{%
      The macro \string\mycmd\space got \the\arg@count\space
       arguments,\MessageBreak
      but expected are 2 arguments.\MessageBreak
      \@ehd
    }%
  \fi
  \edef\process@me{%
    \noexpand\@res
    {\etex@unexpanded\expandafter{\argi}}%
    {\etex@unexpanded\expandafter{\argii}}%
  }%
  \process@me
}
\newcommand{\@res}[2]{%
  \ensuremath\left( #1 - #2 \right)^{-1}
}
\makeatother

\makeatletter
\newcount\arg@count
\makeatother

\allowdisplaybreaks
\numberwithin{equation}{section}

\theoremstyle{definition}
\newtheorem{de}{Definition}[section]

\theoremstyle{plain}

\newtheorem{definition}[de]{Definition}
\newtheorem{assumption}[de]{Assumption}
\newtheorem{proposition}[de]{Proposition}
\newtheorem{lemma}[de]{Lemma}
\newtheorem{theorem}[de]{Theorem}

\newtheorem{corollary}[de]{Corollary}

\theoremstyle{remark}
\newtheorem{remark}[de]{Remark}

\newcommand{\Acal}{\mathcal{A}}
\newcommand{\Bcal}{\mathcal{B}}

\newcommand{\Ecal}{\mathcal{E}}

\newcommand{\Hcal}{\mathcal{H}}
\newcommand{\Ical}{\mathcal{I}}

\newcommand{\Lcal}{\mathcal{L}}

\newcommand{\Rcal}{\mathcal{R}}

\newcommand{\Xcal}{\mathcal{X}}

\newcommand{\Sbb}{\mathbb{S}}

\DeclareMathOperator{\essran}{ess\,ran}
\newcommand{\di}{\mathrm{d}}

\newcommand{\ONE}{\mathbbm{1}}

\def\XXint#1#2#3{{\setbox0=\hbox{$#1{#2#3}{\int}$}
\vcenter{\hbox{$#2#3$}}\kern-.5\wd0}}

\usepackage{multirow,booktabs}
\usepackage{scrextend}
\usepackage[outline]{contour}
\usepackage[ruled, vlined]{algorithm2e}

\SetCommentSty{mycommfont}
\usepackage{soul}
\usepackage{sidecap}
\usepackage{tabularx}
\usepackage{amsmath}
 \pgfplotsset{every tick label/.append style={font=\footnotesize},
}

\graphicspath{{Pictures/}}

\numberwithin{equation}{section}

\renewcommand{\eps}{\varepsilon}
\renewcommand{\phi}{\varphi}

\newcommand{\p}{\partial}

\title{A uniform ergodic theorem for degenerate flows on the annulus}

\author{Jonathan Ben-Artzi}
\author{Baptiste Morisse}

\address{School of Mathematics, Cardiff University, Cardiff CF24 4AG, Wales, UK}

\email{Ben-ArtziJ@cardiff.ac.uk}
\email{MorisseB@cardiff.ac.uk}

\date{\today}



\thanks{The authors acknowledge support from the Early Career Fellowship EP/N020154/1 of the UK's Engineering and Physical Sciences Research Council (EPSRC)}

\begin{document}
\maketitle

\begin{abstract}
Motivated by the well-known phase-space portrait of the nonlinear pendulum, the purpose of this paper is to obtain convergence rates in the ergodic theorem for flows in the plane that have arbitrarily slow trajectories. Considering bounded periodic trajectories near the heteroclinic orbits, it is shown that despite lacking a spectral gap, there exists a functional space (which is a strict subset of $L^2$) on which time averages converge uniformly to spatial averages (with an explicit rate). The main ingredient of the proof is an estimate of the density of the spectrum of the generator of the flow near zero.
\end{abstract}

\hypersetup{
  pdfauthor = {},
  pdftitle = {},
  pdfsubject = {},
  pdfkeywords = {}
}

\setcounter{tocdepth}{1}

\section{Introduction}\label{sec:intro}
The nonlinear pendulum is a fundamental   dynamical system, its phase portrait described qualitatively in Figure \ref{fig:pendulum}. All solutions are periodic with the exception of the heteroclinic orbits connecting the unstable equilibria. These orbits separate \emph{bounded} solutions (for which the pendulum does not complete a full revolution) from the \emph{unbounded} solutions (where the pendulum always revolves in one direction, never stopping). This provides a decomposition into distinct invariant sets. See  \cite{Moser2005} for further details.
	\begin{figure}[H]
	\includegraphics[width=10cm,height=5cm]{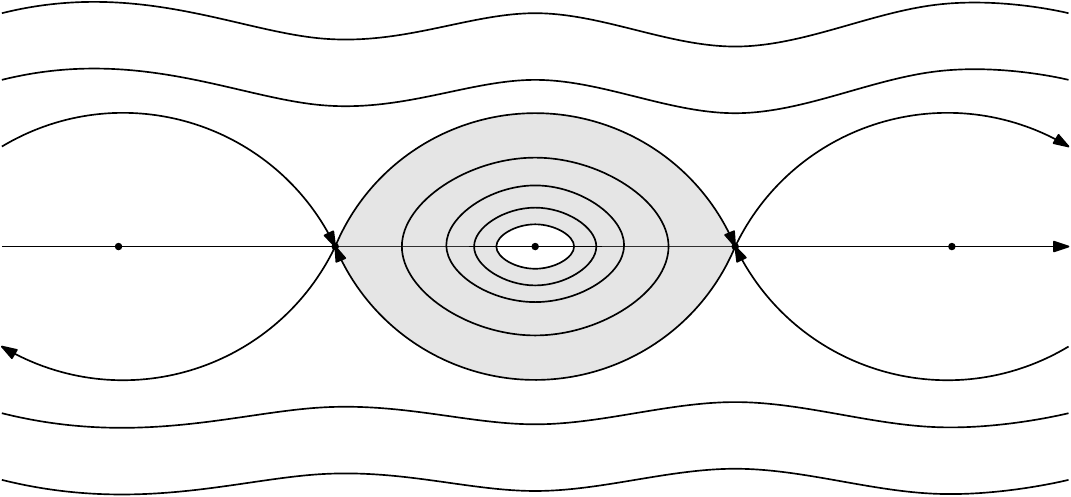}
	\caption{Qualitative description of the phase-space of the nonlinear pendulum. The shaded region has an outer boundary comprised of two heteroclinic orbits, and a small bounded periodic solution as an inner boundary.}\label{fig:pendulum}
	\end{figure}
Let $\Phi:\R\times\R^2\to\R^2$ be the associated evolution function, satisfying the standard group structure: $\Phi(0,x)=x$, and $\Phi(t_2,\Phi(t_1,x))=\Phi(t_2+t_1,x)$. The mean ergodic theorem \cite{VonNeumann1932a} states that for any $f\in  L^2(\R^2)$, the time average $\frac{1}{2T}\int_{-T}^Tf(\Phi(t,\cdot))\,\di t$ converges, as $T\to+\infty$, to the projection onto the invariant subspace, i.e. functions constant along each trajectory.  Since this system has no spectral gap, this convergence lacks a rate in $L^2$. The purpose of this paper is to seek subspaces $\Xcal\subset L^2(\R^2)$ on which there \emph{is} a rate (and find the rate).
	
The absence of a spectral gap makes this problem challenging. Let us mention three distinct difficulties. The first comes from the small amplitude solutions around the steady equilibria (`centers') where $\Phi$ has a fixed point. The second comes from the unbounded solutions which are infinitely long. Lastly, the third difficulty is due to the bounded solutions near the heteroclinic orbits which become arbitrarily slow. These three types of trajectories have one thing in common: they all contribute to the spectrum of the infinitesimal generator of the flow near $0$, thereby denying it of a spectral gap. \\

In this paper we concentrate on the third aforementioned difficulty:  bounded solutions that become arbitrarily slow. These solutions are highlighted in the shaded region in Figure \ref{fig:pendulum}: the flow becomes arbitrarily slow near the heteroclinic orbits which lie on the outer boundary (though the speed along each trajectory isn't constant).  We simplify the geometry by considering a circular annulus $\Acal=[0,1]_m\times\Sbb_\theta$ where near one of the boundaries the flow becomes arbitrarily slow  (and, for simplicity, constant along each trajectory). Here $m\in[0,1]$ parametrizes the radial variable, where $m=0$ corresponds to the heteroclinic (infinitely slow) orbit, and $\theta\in\Sbb=[0,1]_{\mathrm{per}}$ is the angular variable.
For convenience, we  reverse the direction of the flow, so that we consider the operator $\frac{\di}{\di\theta}$ rather than $-\frac{\di}{\di\theta}$.   The speed of the flow is given by  $\varphi(m)\geq0$ which is assumed to satisfy $\phi(m)\to0$ as $m\to 0$ (see Figure \ref{fig:annulus}).

	\begin{figure}[H]
	\includegraphics[scale=0.3]{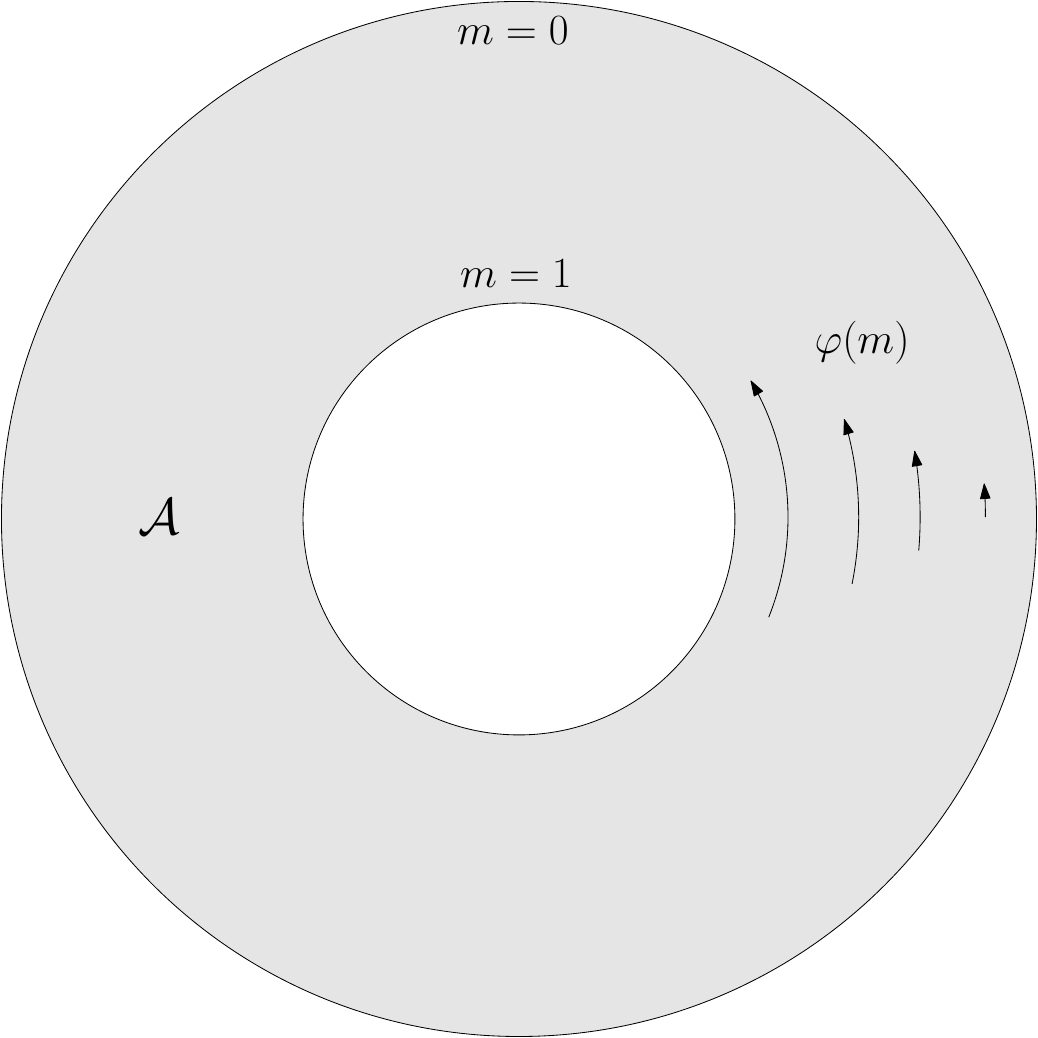}
	\caption{The annulus $\Acal=[0,1]_m\times\Sbb_\theta$. The steady circular flow has speed $\varphi(m)\geq0$, depending on the radial variable $m$, which tends to $0$ as $m\to0$.}\label{fig:annulus}
	\end{figure}

The  infinitesimal generator of the flow is given by the  self-adjoint operator
	\begin{equation}\label{eq:operator-a}
	A
	=
	-i\phi(m)\frac{\p}{\p\theta}
	\qquad
	\text{acting in }L^2(\Acal).
	\end{equation}
The corresponding time averages are defined as
	\begin{equation*}\label{eq:pt}
	P^Tf:=\frac{1}{2T}\int_{-T}^Te^{itA}f\,\di t=\frac{1}{2T}\int_{-T}^Tf(m,\theta+t\phi(m))\,\di t
	\end{equation*}
and the spatial average is defined as (noting that, without loss of generality, the lengths of all trajectories are normalized to $1$)
	\begin{equation*}\label{eq:p}
	Pf:=\int_{\Sbb}f(m,\theta)\,\di\theta.
	\end{equation*}
In Theorem \ref{thm:main} below we identify a functional subspace $\Xcal\subset L^2(\Acal)$ on which  $\|P^T-P\|_{\Xcal\to L^2(\Acal)}\to0$ as $T\to+\infty$ with an explicit rate. To state this theorem precisely, we must first introduce some functional spaces in order to define the subspace $\Xcal$. We write $L^2(\Acal)=L^2\left([0,1] ; L^2(\Sbb)\right)$ and hence represent elements as
	\begin{equation*}
	f=\int_{[0,1]}^\oplus f_m\,\di m\qquad\text{with}\qquad f_m \in L^2(\Sbb).
	\end{equation*}
We define the Fourier modes of each fiber $f_m$  as
	\begin{equation*}
	\widehat{f_m}(k) := \int_{0}^1 f_m(\theta) e^{- 2\pi ik \theta } \,\di\theta, \qquad\ k\in\Z.
	\end{equation*}
Let $L^2_0(\Sbb)\subset L^2(\Sbb)$ be the subspace of functions with zero average: $L^2_0(\Sbb) = \left\{ g \in L^2(\Sbb) \,:\, \widehat{g}(0) = 0 \right\}.$ The standard homogeneous Sobolev spaces on $L^2_0(\Sbb)$ are defined as:
	\[
	\dot{H}^{\gamma} := \left\{ g\in L^2_0(\mathbb{S}) \, : \, \|g\|_{ \dot{H}^{\gamma} }^2:=\sum_{k\in\Z}|k|^{2\gamma} |\widehat{g}(k)|^2<+\infty \right\},\qquad \gamma \geq 0.
	\]
Turning to regularity in the $m$ variable, we define for $s\in(0,1)$ and $\gamma \geq 0$ the space
	\[
	\dot{\mathcal{H}}^{s,\gamma}
	:=
	W^{s,2}\left([0,1] ; \dot{H}^{\gamma} \right)
	=
	\left\{ f \in L^2\left( [0,1] ; \dot{H}^{\gamma} \right) \, : \,
	 \int_0^1 \int_0^1 \frac{\|f_m - f_{m'}\|^2_{\dot{H}^{\gamma}} }{|m-m'|^{1+2s}} \,\di m\, \di m'<+\infty \right\}.
	\]
We endow this space with the  norm
	\begin{equation*}
	\label{def.norm}
	\|f\|_{s,\gamma}^2
	:=
	\|f\|_{L^2\left([0,1]; \dot{H}^{\gamma}\right)}^2 + \int_0^1 \int_0^1 \frac{\|f_m - f_{m'}\|^2_{\dot{H}^{\gamma}} }{|m-m'|^{1+2s}} \,\di m\, \di m'.
	\end{equation*}

For further discussion of these spaces we refer to  \cite[Theorem 10.2]{Lions1972} and \cite[Corollary 26]{Simon1990}.
We can finally define our main functional space:
\begin{definition}[The functional space $\dot{\mathcal{H}}^{s,\gamma}_0$]
	\label{definition.H}
	Let $s>1/2$ and $\gamma\geq0$.
	We define the subspace $\dot{\mathcal{H}}^{s,\gamma}_0\subset \dot{\mathcal{H}}^{s,\gamma}$ to be the set of functions that vanish along $\{m=0\}$:
	\begin{equation*}\label{eq:main-space}
		\dot{\mathcal{H}}^{s,\gamma}_0:
		=
		\left\{ f \in \dot{\mathcal{H}}^{s,\gamma}  \, : \, f_0 = 0 \right\} .
	\end{equation*}
We equip this subspace with the norm $\left\| \cdot\right\|_{s,\gamma}$.
\end{definition}

\bigskip

We are now ready to state our main result, starting with the assumption that the flow degenerates near $m=0$.

\begin{assumption}\label{ass:main}
There exist $m_0,c,\alpha>0$ such that $\phi(m)=cm^\alpha$ on $[0,m_0)$ and  $\phi$ is continuous and bounded uniformly away from $0$ on $[m_0,1]$. \end{assumption}

\begin{theorem}\label{thm:main}
Let Assumption \ref{ass:main} hold. 
Then there exist  $s > 1/2$ and $\gamma \geq 0$  satisfying
	\begin{equation}
	\label{constraint}
	\gamma + \frac{s}{\alpha} > \frac12
	\end{equation}
and a constant $C>0$ such that the following uniform rate holds:
	\begin{equation*}\label{eq:conv-rate-frac}
	\|P^T\|_{\dot{\mathcal{H}}^{s,\gamma}_0\to L^2(\Acal)}\leq CT^{-\ell/2},\qquad\forall T>1,
	\end{equation*}
where $\ell=\min\{\frac{2s}{\alpha}-\varepsilon,2\}$ for any $\varepsilon>0$ and where $C$  does not depend on $T$.
\end{theorem}

\begin{remark}
Observe that the subspace $\dot{\mathcal{H}}^{s,\gamma}_0$ includes only functions with  average zero along any fiber, so that $P$ is trivial on this subspace:  $Pf=0$ for any $f\in \dot{\mathcal{H}}^{s,\gamma}_0$. Therefore  $\|P^T\|_{\dot{\mathcal{H}}^{s,\gamma}_0\to L^2(\Acal)}=\|P^T-P\|_{\dot{\mathcal{H}}^{s,\gamma}_0\to L^2(\Acal)}$.
\end{remark}

 \begin{remark}
Assumption \ref{ass:main} can be weakened, as only the behavior of $\phi$ near its zeros is important. For instance, all our arguments below could be modified to handle a case where $\phi$ has  several zeros where it vanishes like $|m-m_i|^{\alpha_i}$ for some $\alpha_i>0$ near each zero $m_i$. Furthermore, the requirement that $\phi$ be continuous can also be weakened. However these are technicalities that we do not pursue here.\\
 \end{remark}

For any self-adjoint operator $A:D(A)\subset\Hcal\to\Hcal$ acting in a separable Hilbert space $\Hcal$, von Neumann's ergodic theorem \cite{VonNeumann1932a} guarantees the strong convergence (i.e without a rate)  $P^T\to P$ in $\Hcal$. Von Neumann's idea was to use the spectral theorem to write $A=\int_\R\lambda\,\di E(\lambda)$, where $\{E(\lambda)\}_{\lambda\in\R}$ is the resolution of the identity of  $A$. This leads to
	\begin{equation}
	\begin{split}
	(P^T-P)f
	&=
	\frac{1}{2T}\int_{-T}^Te^{itA}f\,\di t-Pf
	=
	\frac{1}{2T}\int_{-T}^T\int_\R e^{it\lambda}\,\di E(\lambda)f\,\di t-Pf\\
	&=
	\frac{1}{2T}\int_{-T}^T\int_{\R\setminus\{0\}} e^{it\lambda}\,\di E(\lambda)f\,\di t
	=
	\int_{\R\setminus\{0\}} \frac{\sin T\lambda}{T \lambda}\,\di E(\lambda)f.\label{eq:von-neumann}
	\end{split}
	\end{equation}
The last expression tends to $0$ as $T\to+\infty$. It is a simple consequence that   a spectral gap leads to a rate of convergence of $T^{-1}$. However, as we prove in Corollary \ref{cor:spec-gap} below, the operator $A$ defined in \eqref{eq:operator-a} has no spectral gap due to Assumption \ref{ass:main}.

We overcome this difficulty by using  the results of \cite{Ben-Artzi2019c}, where it is shown that even when there's no spectral gap, one can still extract a rate (albeit perhaps slower and only on a subspace) if the density of the spectrum near $0$ is bounded. More precisely, if there exists a subspace $\Xcal\subset L^2(\Acal)$ and some $r>0$ such that the density of states (DoS) of $A$ has a bound of the form
	\begin{equation}\label{eq:dos-bound}
	\left|\frac{\di}{\di\lambda}\left(E(\lambda)f,g\right)_{L^2(\Acal)}\right|
	\leq
	\psi(\lambda)\|f\|_\Xcal\|g\|_\Xcal,\qquad\forall f,g\in\Xcal,\,\forall\lambda\in (-r,r)\setminus\{0\},
	\end{equation}
where $\psi\in L^1(-r,r)$ is strictly positive a.e. on $(-r,r)$. This then allows one to replace the integration $\di E(\lambda)$ in \eqref{eq:von-neumann} with $\frac{\di E(\lambda)}{\di\lambda}\di\lambda$ and to apply the  estimate \eqref{eq:dos-bound} near $\lambda=0$. Therefore, to prove Theorem \ref{thm:main} the main task is to obtain an estimate of the form \eqref{eq:dos-bound}, which involves identifying an appropriate subspace $\Xcal$. This is achieved thanks to the observation that the operator $A$ is unitarily equivalent to the multiplication operator $\phi k$ via a Fourier transform in $\theta$ (where $k\in\Z$ is the Fourier conjugate of $\theta$). It is worthwhile mentioning that the result \cite{Ben-Artzi2019c} has been further refined recently in \cite{Kachurovskii2023a,Kachurovskii2023}.

To obtain an estimate of the DoS as in \eqref{eq:dos-bound}, the first step is to understand the structure of the spectrum. Here we take the point of view that $A$ is fibered in $m$,  composed of the one-dimensional operators
	\begin{equation*}\label{eq:fibers}
	A(m) = -i\phi(m) \frac{\di }{\di \theta}\qquad\forall m\in[0,1],
	\end{equation*}
acting on the circle $\Sbb$. This point of view is common in the context of the  Euler equations, see  \cite{Cox2013a} for instance. The main difficulty is now evident: the spectrum of each fiber $A(m)$ is discrete, while the spectrum of $A$ may have discrete, absolutely continuous and singular continuous parts.\\

This approach has proven useful in a similar context before, where in \cite{Ben-Artzi2013d} a rate of convergence was obtained for shear flows. The result \cite{Ben-Artzi2013d} treats unbounded flows which are of a similar nature to the unbounded solutions appearing in Figure \ref{fig:pendulum}. As already mentioned, the pendulum is a fundamental dynamical system, and gaining a better understanding has many applications, cf.  \cite{Ben-Artzi2014} for a brief discussion on the relationship to kinetic theory. With our present result, a full understanding of convergence rates for the ergodic theorem associated to the pendulum is only lacking an understanding of the contribution of the center.
Finally, we note that a closely related (though not the same) problem  is that of \emph{mixing}, which has seen a flurry of activity in recent years. We point out \cite{Yao2014} as one example.	\\

In Section \ref{sec:fibered} we discuss properties of fibered self-adjoint operators.  Then, in Section \ref{sec:annulus} we turn our attention to the flow in an annulus, first proving a bound on the DoS (Proposition \ref{thm:main2}) and subsequently obtaining a rate for the associated ergodic theorem, proving Theorem \ref{thm:main}.

\section{Spectral analysis of  fibered operators}\label{sec:fibered}
In this section we recall some properties of self-adjoint operators and of self-adjoint fibered operators, and prove some results which are not always readily available in the standard literature. We also discuss a simple illustrative example where we estimate the DoS. This shall prove useful later, when we  prove Theorem \ref{thm:main}. Our discussion in this section remains as general as possible, working with abstract self-adjoint operators in abstract Hilbert spaces.\\

\subsection{Spectral Theory and Fibered Operators.}
Since the spectral theorem plays an essential role in our proof, it is worthwhile recalling the definition of the resolution of the identity of a self-adjoint operator.
Let $\Hcal$ be a separable Hilbert space, and let $H:D(H)\subset\Hcal\to\Hcal$ be a self-adjoint operator. Its associated spectral family $\{E(\lambda)\}_{\lambda\in\R}$  is a family of projection operators in $\Hcal$ with the property that, for each $\lambda\in\R$, the subspace $\Hcal^\lambda=E(\lambda)\Hcal$ is the largest closed subspace such that
\begin{enumerate}
\item $\Hcal^\lambda$ \emph{reduces} $H$, namely, $HE(\lambda)g=E(\lambda)Hg$ for every $g\in D(H)$. In particular, if $g\in D(H)$ then also $E(\lambda)g\in D(H)$.
\item $(Hu,u)_{\Hcal}\leq \lambda (u,u)_{\Hcal}$ for every $u\in\Hcal^\lambda\cap D(H)$.
\end{enumerate}

Given any $f,g\in \Hcal$ the spectral family defines a complex function of bounded variation on the real line, given by
	\begin{equation*}\label{eq:spec-meas}
	\R\ni\lambda\mapsto(E(\lambda)f,g)_{\Hcal}\in\C.
	\end{equation*}
It is well-known that such a function gives rise to a complex measure (depending on $f,g$) called the \emph{spectral measure}. Recall the following useful fact:
\begin{proposition}[{\cite[X-\S1.2, Theorem 1.5]{Kato1995}}]\label{prop:abs-cont-subs}
Let $U\subset\R$ be open. The set of $f,g\in\Hcal$ for which the spectral measure is absolutely continuous in $U$ with respect to the Lebesgue measure forms a closed subspace $\mathcal{AC}_U\subset\Hcal$. This subspace is referred to as the \emph{absolutely continuous subspace of $H$ on $U$}.
\end{proposition}

We can go even further, and try to differentiate the spectral measure.
Let $\mathcal{AC}_U\subset\Hcal$ be the absolutely continuous subspace of $H$ on $U$ and let $\lambda\in U$.  Suppose that there exists a subspace $\mathcal X\subset\mathcal{AC}_U$ equipped with a stronger norm such that the bilinear form $\frac{\di}{\di\lambda}(E(\lambda)\cdot,\cdot)_{\Hcal}:\mathcal X \times \mathcal X\to\C$ is bounded.  Then this bilinear form   is called the \emph{density of states (DoS)} of $H$ at $\lambda$. It obviously depends on the choice of subspace $\Xcal$. In the context of Schr\"odinger operators, common examples of such subspaces include weighted-$L^2$ spaces and $L^p$ spaces, cf. \cite{Ben-Artzi2019c}. In physics, the DoS represents the number possible states a system can attain at the energy level $\lambda$. It is worthwhile noting that an alternative approach for estimating the DoS is via the so-called \emph{limiting absorption principle}, see the classical result \cite{Agmon1975} for instance.\\

We are now ready to discuss fibered operators. We follow the notation of \cite[p. 283]{Reed1978}. Let $\Hcal'$ be a Hilbert space and let $(M,\di\mu)$ be a measure space. Let $A(\cdot):M\to\Lcal^{s.a.}(\Hcal')$ be a measurable function taking values in the space of self-adjoint operators (not necessarily bounded) on $\Hcal'$ (with appropriate domains). Let	
	\begin{equation}\label{eq:fibered-h-a}
	\Hcal=\int_M^\oplus\Hcal'\qquad\text{and}\qquad A=\int_M^\oplus A(m)\di\mu(m).
	\end{equation}
It is well-known that since all $A(m)$ are self-adjoint, so is $A$. Given  an element $f\in\Hcal$, we denote its fibers as $f_m\in\Hcal'$ so that 
	\begin{equation}\label{eq:fibered-f}
	f=\int_M^\oplus f_m\,\di\mu(m).
	\end{equation}
We   denote the resolution of the identity of $A$ by $\{E(\lambda)\}_{\lambda\in\R}$ and of $A(m)$ by $\{E_m(\lambda)\}_{\lambda\in\R}$.

\begin{lemma}\label{lem:fiber}
The resolutions of the identity  satisfy the natural decomposition
	\begin{equation*}
	E(\lambda)=\int_{M}^\oplus E_m(\lambda)\,\di\mu(m).
	\end{equation*}
\end{lemma}

\begin{proof}
By standard functional calculus, we apply the characteristic function $\ONE_{(-\infty,\lambda_0]}$ to
	\begin{equation*}
	\int_\R\lambda\,\di E(\lambda)=A=\int_M^\oplus A(m)\,\di\mu(m)=\int_M^\oplus \int_\R\lambda\,\di E_m(\lambda)\,\di\mu(m)
	\end{equation*}
to obtain the assertion of the lemma (at the point $\lambda_0$).
\end{proof}

The spectrum $\sigma(A)$ is characterized as follows:
	\begin{equation}\label{eq:main-relation}
	\lambda\in\sigma(A)
	\qquad\Leftrightarrow\qquad
	\forall\eps>0,\quad\mu\left(\left\{m\;:\;\sigma(A(m))\cap(\lambda-\eps,\lambda+\eps)\neq\emptyset\right\}\right)>0.
	\end{equation}
An immediate consequence of this is:
	\begin{equation}\label{eq:spec-containment}
	\sigma(A)\subset\overline{\bigcup_m\sigma(A(m))}.
	\end{equation}
The following characterization of eigenvalues follows:
	\begin{equation}\label{eq:eigval}
	\mu\left(\left\{m\;:\;\lambda\text{ is an eigenvalue of }A(m)\right\}\right)>0
	\qquad\Rightarrow\qquad
	\lambda\text{ is an eigenvalue of }A.
	\end{equation}

\begin{proposition}\label{prop:cont}
Assume that  the measure $\di\mu$  is the Borel measure associated to some given topology and that $M$ is compact. Assume that $\Sigma:M\to\mathrm{clos}(\R)=\{$closed subsets in $\R\}$ given by $\Sigma(m)=\sigma(A(m))$   is continuous (we take the Hausdorff distance on $\mathrm{clos}(\R)$). Then $\sigma(A)=\cup_m\sigma(A(m))$.
\end{proposition}

\begin{proof}
Since $M$ is compact and $\Sigma$ is continuous, we have that $\cup_m\sigma(A(m))=\overline{\cup_m\sigma(A(m))}$. Hence, considering \eqref{eq:spec-containment} we only need to prove that $\sigma(A)\supset\cup_m\sigma(A(m))$. Suppose that $\lambda\in\cup_m\sigma(A(m))=\cup_m\Sigma(m)$. In particular there exists some $m_0\in M$ such that $\lambda\in\Sigma(m_0)$. By the continuity of $\Sigma$, for each $\eps>0$ there exists an open neighborhood $U$ of $m_0$ such that $\Sigma(m)\cap(\lambda-\eps,\lambda+\eps)\neq\emptyset$ for all $m\in U$, which, by \eqref{eq:main-relation}, implies that $\lambda\in\sigma(A)$ (note that $U$ has positive measure).
\end{proof}

\subsection{Illustrative example}\label{sec:toy}
The proof of Theorem \ref{thm:main} will rely on the ideas contained in the  following discussion, and in particular  in the proof of  Proposition \ref{thm:toy}. Specifically, here we study how the DoS of a fibered operator $A$ depends upon the behavior of eigenvalues of the fibers $A(m)$ when there are finitely-many contributing eigenvalues. In the proof of Theorem \ref{thm:main} there will be infinitely-many such eigenvalues.
Consider the following fibered problem. Let $\left(M,\di \mu\right) = \left([0,1],\di m\right)$ (Lebesgue measure) with the usual Borel $\sigma$-algebra,  let $\Hcal=\int_{[0,1]}^\oplus\Hcal'$ and  $A=\int_{[0,1]}^\oplus A(m)\di m$, and assume that (see Figure \ref{fig:toy} for an illustration):

\begin{assumption} \label{ass:ecal-1}
For all $m\in [0,1]$ the spectrum of $A(m)$ in the energy band $I:=(a,b)\subset\R$  is given by a single eigenvalue $\Ecal(m)$  (with multiplicity $1$) depending on $m$. Moreover, the function $\Ecal:[0,1]\to I$ is  measurable.
\end{assumption}

It is well-known that $\sigma(A)\cap I=\essran(\Ecal)$ and if $\Ecal$ is continuous then $\essran$ may be replaced by $\ran$ (see Proposition \ref{prop:cont}). Moreover, if $\Ecal$ is constant  on some open set $U$ and equal to $\overline{\Ecal}$ there, then $\overline{\Ecal}$ is an eigenvalue of $A$ (see \eqref{eq:eigval}). These two statements combined indicate that one can easily construct examples with eigenvalues embedded in the essential spectrum (or at its boundary). Hence  even in this simple example, the absolute continuity of the spectrum of $A$ depends on the nature of the function $\Ecal$. Let us assume that there are no embedded eigenvalues:

\begin{assumption}\label{ass:ecal}
$\Ecal(m)$ is $C^1$ and is not  constant on sets of positive measure.
\end{assumption}

	\begin{figure}[H]
	\includegraphics[scale=0.6]{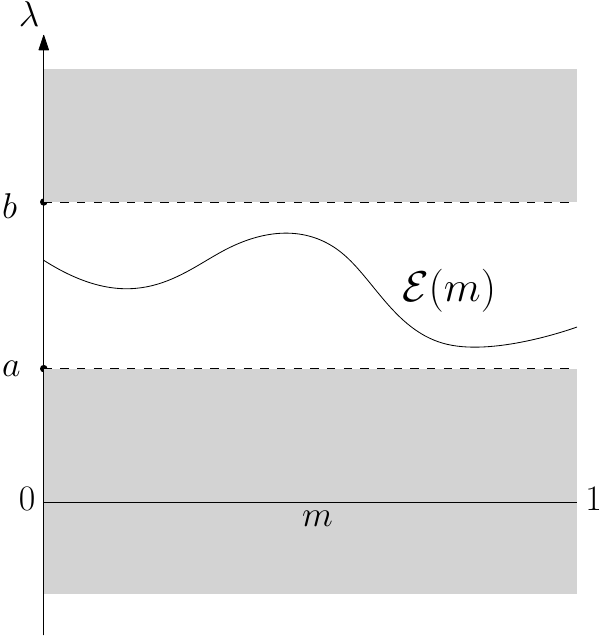}
	\caption{In this illustrative example, for each fiber $m$, there is one eigenvalue in the energy band $I=(a,b)$: $\sigma(A(m))\cap(a,b)=\Ecal(m)$. Outside of this energy band  there are no conditions on the spectrum.}\label{fig:toy}
	\end{figure}

\begin{definition}
Let $\lambda\in\ran(\Ecal)$ and denote $M_{\lambda}:=\Ecal^{-1}(\lambda)\subset [0,1]$. We say that $\lambda$ is a regular value if for any $m\in M_\lambda$, $\Ecal'(m)\neq 0$. We also define the set
	\begin{equation*}
	\sigma_{\mathrm{reg},I}(A):=\left\{\lambda\in I\;:\;\lambda\in\ran(\Ecal)\text{ is a regular value}\right\}.
	\end{equation*}
\end{definition}

\begin{lemma}
$\sigma_{\mathrm{reg},I}(A)$ is an open subset of $I$.
\end{lemma}

\begin{proof}
The claim is almost trivial, as the condition of being a regular value is an open condition. We only note that any point in the set $I\cap\partial(\ran(\Ecal))$ cannot be a regular value since $I$ is open.
\end{proof}

%
%
%

\begin{proposition}\label{thm:toy}
Let $\Hcal$ and $A$ be as in \eqref{eq:fibered-h-a} with $\left(M,\di \mu\right) = \left([0,1],\di m\right)$. Suppose that $\Ecal$ satisfies Assumptions \ref{ass:ecal-1} and \ref{ass:ecal}. Let $\lambda\in\sigma_{\mathrm{reg},I}(A)$. Then the DoS of $A$ at $\lambda$ satisfies the estimate
	\begin{equation}\label{eq:est-sing-eig}
	\left|\frac{\di}{\di\lambda}(E(\lambda)f,g)_{\Hcal}\right|
	\leq
	\left( \sum_{m\in M_{\lambda}}\frac{1}{|\Ecal'(m)|} \right) \, \|f\|_{\Hcal}\|g\|_{\Hcal},\qquad\forall f,g\in\Hcal.
	\end{equation}
\end{proposition}

\begin{proof}
Since $M=[0,1]$ is compact, $\Ecal\in C^1$, and $\lambda$ is a regular value, $M_{\lambda}$ is a finite set and we denote its elements $M_{\lambda}=\{m_1,\dots,m_k\}$. If   $\Ecal(m)\leq\lambda$ then the projection operator $E_m(\lambda)$ may be represented as $P_m(\Ecal(m))+ E_m(a)$ where $P_m(\Ecal(m))$ is the projection onto the eigenspace in $\Hcal'$ (the $m$th ``copy'') corresponding to the eigenvalue $\Ecal(m)$. If  $\Ecal(m)>\lambda$ then the projection operator $E_m(\lambda)$ is equal to the projection operator  $E_m(a)$. Letting $f,g\in\Hcal$, with fibers $f_m,g_m\in \Hcal'$ (as in \eqref{eq:fibered-f}), we therefore have
	\begin{align*}
	(E(\lambda)f,g)_{\Hcal}
	&=
	\int_{[0,1]}\left(E_m(\lambda)f_m,g_m\right)_{\Hcal'}\,\di m\\
	&=
	\int_{\{m\;:\;\Ecal(m)\leq\lambda\}}\left(P_m(\Ecal(m))f_m,g_m\right)_{\Hcal'}\,\di m+\int_{[0,1]}\left(E_m(a)f_m,g_m\right)_{\Hcal'}\,\di m.
	\end{align*}
Upon differentiation in $\lambda$ the second term on the right hand side is eliminated, and one is left with
	\begin{align*}
	\frac{\di}{\di\lambda}(E(\lambda)f,g)_{\Hcal}
	&=
	\lim_{h\downarrow0}\left[\frac{1}{h}\int_{\{m\;:\;\lambda<\Ecal(m)\leq\lambda+h\}}\left(P_m(\Ecal(m))f_m,g_m\right)_{\Hcal'}\,\di m\right]\\
	&=
	\sum_{i,\; \Ecal'(m_i)>0}\lim_{h\downarrow0}\left[\frac{1}{h}\int_{(m_i, m_i+\frac{h}{\Ecal'(m_i)}]}\left(P_m(\Ecal(m))f_m,g_m\right)_{\Hcal'}\,\di m\right]\notag\\
	&\quad+
	\sum_{i,\; \Ecal'(m_i)<0}\lim_{h\downarrow0}\left[\frac{1}{h}\int_{(m_i+\frac{h}{\Ecal'(m_i)}, m_i]}\left(P_m(\Ecal(m))f_m,g_m\right)_{\Hcal'}\,\di m\right].
	\end{align*}
Making the change of variables $\eta_i=\frac{h}{|\Ecal'(m_i)|}$ we have
	\begin{align*}
	\frac{\di}{\di\lambda}(E(\lambda)f,g)_{\Hcal}
	&=
	\sum_{i,\; \Ecal'(m_i)>0}\frac{1}{\Ecal'(m_i)}\lim_{\eta_i\downarrow0}\left[\frac{1}{\eta_i}\int_{(m_i, m_i+\eta_i]}\left(P_m(\Ecal(m))f_m,g_m\right)_{\Hcal'}\,\di m\right]\\
	&\quad+
	\sum_{i,\; \Ecal'(m_i)<0}\frac{1}{|\Ecal'(m_i)|}\lim_{\eta_i\downarrow0}\left[\frac{1}{\eta_i}\int_{(m_i-\eta_i, m_i]}\left(P_m(\Ecal(m))f_m,g_m\right)_{\Hcal'}\,\di m\right]\\
	&=
	\sum_{i=1}^k\frac{1}{|\Ecal'(m_i)|}\left(P_{m_i}(\Ecal(m_i))f_{m_i},g_{m_i}\right)_{\Hcal'}.
	\end{align*}
Note that one must be careful if one of the $m_i$ is  $0$ (resp. $1$) and $\Ecal'(0)<0$ (resp. $\Ecal'(1)>0$) as then the above argument requires a slight adjustment. However the same conclusion holds. A simple use of the Cauchy-Schwartz inequality and the orthogonality properties of the projection operators leads to the desired estimate \eqref{eq:est-sing-eig}.
\end{proof}

\section{Proof of the Main Theorem}\label{sec:annulus}
The proof of Theorem \ref{thm:main} follows by first obtaining an estimate on the Density of States (DoS) of $A$ (Proposition \ref{thm:main2}), and then invoking a result from \cite{Ben-Artzi2019c} on the relationship between the DoS near $0$ and the rate of convergence in the ergodic theorem (Proposition \ref{thm:abstract}). We begin with a few results concerning our functional spaces.

\begin{proposition}[{\cite[Corollary 26]{Simon1990}}]\label{prop:simon}
	\label{prop:embedding}
	Let $s>1/2$ and $\gamma\geq0$. Any $f\in\dot{\mathcal{H}}^{s,\gamma}$ is $(s-1/2)$-H\"older continuous with respect to the $m$ variable.
	That is,  there is a constant $C(s)$  such that for any $f\in\dot{\mathcal{H}}^{s,\gamma}$ and any $m, m' \in [0,1]$
	\begin{equation}
		\label{ineq.Holder}
		\left\|f_m - f_{m'} \right\|_{ \dot{H}^{\gamma} } \leq C(s) \left\| f \right\|_{s,\gamma} |m-m'|^{s-1/2}
	\end{equation}
where $C(s)$  depends neither on $\gamma$ nor on $f$.
\end{proposition}

Consequently, it follows that for all $f\in \dot{\mathcal{H}}^{s,\gamma}_0$
	\begin{equation}
		\label{ineq.Holder.0}
		\|f_m\|_{\dot{H}^{\gamma}} \leq C(s) \left\| f \right\|_{s,\gamma} m^{s-1/2}, \qquad \forall  m\in [0,1],
	\end{equation}
since $f_0 = 0$.
Next, we  state  a result on the H\"older regularity of the Fourier coefficients of functions in $\dot{\mathcal{H}}^{s,\gamma}$, which is just a corollary of the H\"older continuity of $f$ stated in Proposition \ref{prop:embedding}:
\begin{lemma}
	\label{prop.Holder.Fourier}
	For any $f\in\dot{\mathcal{H}}^{s,\gamma}_0$, there holds
	\begin{equation}
		\label{ineq.Holder.Fourier}
		\left| \widehat{f}_m(k) - \widehat{f}_{m'}(k) \right|
		\leq C
		\|f\|_{s,\gamma}|k|^{-\gamma} |m-m'|^{s-1/2}  , \qquad \forall  m, m'\in[0,1] , \, \forall k \in \mathbb{Z},
	\end{equation}

	\noindent where the constant $C>0$ does not depend on $f$, $k$ or $m$.
\end{lemma}

\begin{proof}
	We readily compute
	\begin{align*}
	\left| \widehat{f}_m(k) - \widehat{f}_{m'}(k) \right|^2
	& = |k|^{-2\gamma} |k|^{2\gamma} \left| \widehat{f}_m(k) - \widehat{f}_{m'}(k) \right|^2 \\
		& \leq |k|^{-2\gamma} \sum_{k' \in \mathbb{Z}} |k'|^{2\gamma} \left| \widehat{f}_m(k') - \widehat{f}_{m'}(k') \right|^2
		 \leq |k|^{-2\gamma} \|f_m - f_{m'}\|_{\dot{H}^{\gamma}}^2.
	\end{align*}
By the H\"older continuity \eqref{ineq.Holder} of $f_m$ in the $\dot{H}^{\gamma}$ norm, we obtain \eqref{ineq.Holder.Fourier}.
\end{proof}

We can now return to our operator $A=-i\phi\frac{\p}{\p\theta}$ which we recall has the following fiber decomposition:
	\begin{equation*}
	-i\phi\frac{\p}{\p\theta}
	=
	A
	=
	\int_{[0,1]}^\oplus A(m)\,\di m
	=
	-i\int_{[0,1]}^\oplus\phi(m) \frac{\di }{\di \theta}\,\di m.
	\end{equation*}

The spectrum of each fiber is clearly discrete:
\begin{lemma}\label{lem:fibers}
The spectrum of each fiber $A(m)$ is given by
	\begin{equation}\label{eq:spec-fiber}
	\sigma(A(m))=\phi(m)\Z.
	\end{equation}
\end{lemma}

\begin{proof}
This is trivially true, due to the unitary equivalence between the operator $-i\frac{\p}{\p\theta}$ and the multiplication operator  $k$ (that is, the operator that multiplies by $k$) where $k\in\Z$.
\end{proof}

But the spectrum of $A$ contains (at least some) continuous parts:
\begin{corollary}\label{cor:spec-gap}
$A$ has no spectral gap. That is, there is $r>0$ such that $(-r,r)\subset\sigma(A)$.
\end{corollary}

\begin{proof}
This follows immediately from Proposition \ref{prop:cont} and Lemma \ref{lem:fibers}, due to Assumption \ref{ass:main}. Note also that $0$ is always an eigenvalue (any function that is constant along trajectories lies in the kernel of $A$).
\end{proof}

It is important to  note that  to each energy $\lambda\in\sigma(A)$ correspond infinitely many fibers.
Indeed, contributions to the spectrum at energy level $\lambda$ 
will come from all $m \in(0,1]$ for which there exists $k\in\Z$ such that $\lambda=k\varphi(m)$. In particular, considering Assumption \ref{ass:main} and \eqref{eq:spec-fiber}, each fiber $m\in(0,m_0]$  will contribute to the  energy levels
	\begin{equation}\label{eq:lambdamk}
	\lambda_{m,k}= ckm^\alpha,\qquad k\in\Z.
	\end{equation}
Conversely, each energy $\lambda\in\R$ will be in the spectrum of any fiber $m\in(0,m_0]$ for which there exists $k\in\Z$ such that
	\begin{equation*}
	m=\left(\frac{\lambda}{ck}\right)^{1/\alpha}.
	\end{equation*}

These observations highlight the fact that the following result, which provides a bound on the density of the spectrum is by no means trivial: this is a continuous spectrum generated from an accumulation  of  eigenvalues belonging to uncountably many fibers.

\begin{proposition}\label{thm:main2}

Let $A=-i\phi(m)\frac{\p}{\p\theta}$ satisfy  Assumption \ref{ass:main}. Assume that $s > 1/2$ and $\gamma \geq 0$  satisfy the constraint \eqref{constraint}.
Then there exists $r>0$ such that the DoS of  $A$ satisfies
	\begin{equation*}
	\left| \frac{\di }{\di \lambda}(E(\lambda)f,g)_{L^2(\Acal)} \right|
	\leq
	C|\lambda|^{\frac{2s}{\alpha} -1} \left\| f \right\|_{s,\gamma} \left\| g \right\|_{s,\gamma} ,\qquad\forall\lambda\in(-r,r)\setminus\{0\},\,\forall f,g\in\dot{\mathcal{H}}^{s,\gamma}_0,
	\end{equation*}
where $C>0$  does not depend on $f$, $g$ or $\lambda$.
\end{proposition}

\begin{remark}
	\label{remark.threshold}
	In the case $\alpha \leq 1$ the constraint \eqref{constraint} is satisfied for all $s > 1/2$ with $\gamma = 0$.
	This means that for $\alpha \leq 1$, the subspace $\dot{\mathcal{H}}^{s,0}_0$ is sufficient to get the estimate of the DoS.
  In the case $\alpha > 1$ however, the constraint \eqref{constraint} is stronger and working in the subspace $\dot{\mathcal{H}}^{s,0}_0$ is not sufficient to obtain estimates of the DoS. This is mainly due to the fact that the eigenvalues $\lambda_{m,k} =   ck m^{\alpha}$ concentrate at $0$ faster as $m\to0$ and $k\to+ \infty$ (see Figure \ref{fig:alpha}).
	To balance that effect, more regularity in the $\theta$ variable is required.

	\begin{figure}[H]
	\centering
	\begin{tikzpicture}[domain=0:3]
		\draw[color=gray,xshift=5cm] (0,0) -- (0,3) -- (3,3) -- (3,0) -- (0,0)
			node at (1.5,-0.5) {$m$}
			node at (-0.1,-0.2) {$0$};
		\draw[xshift=5cm] plot  ( {(\x)^3/9},\x) node at (1.7,3.2) {$\alpha <1$ \qquad};
		\draw[color=gray,xshift=10cm] (0,0) -- (0,3) -- (3,3)-- (3,0) -- (0,0)
			node at (1.5,-0.5) {$m$}
			node at (-0.1,-0.2) {$0$};
		\draw[xshift=10cm] plot (\x , {(\x)^3/9}) node at (1.7,3.2) {$\alpha >1$};
	\end{tikzpicture}
	\caption{The behavior of $m^\alpha$ near $m=0$  for $\alpha<1$ and for $\alpha>1$ is very different, leading to different functional spaces in both cases.}\label{fig:alpha}
	\end{figure}
\end{remark}

\smallskip

\begin{proof}[Proof of  Proposition \ref{thm:main2}]
Without loss of generality we take $\lambda>0$. All arguments can be repeated with minimal (obvious) changes for $\lambda<0$, so we omit this here. Using Lemma \ref{lem:fiber} we write
\begin{align*}
	(E(\lambda)f,g)_{L^2(\Acal)}
	&=
	\int_{[0,1]}\left(E_m(\lambda)f_m,g_m\right)_{L^2(\Sbb)}\,\di m\\
	&=
	\int_{[0,1]}\sum_{\substack{k\in\N\\ \lambda_{m,k}\leq\lambda}}\left(P_m(\lambda_{m,k})f_m,g_m\right)_{L^2(\Sbb)}\,\di m
\end{align*}
where $P_m(\lambda_{m,k})$ is the projection on the Fourier coefficient $\widehat{f_m}(k)$ of $f_m$ (see \eqref{eq:lambdamk} for the definition of $\lambda_{m,k}$).
Defining the energy band
	\begin{equation*}
	\Bcal(\lambda, m ,h) = \left\{ k \in \N \, : \, \lambda < ck m^{\alpha} \leq \lambda + h \right\}
	\end{equation*}
we can express the finite difference
	\begin{align*}
	\Delta_hE(\lambda)
	&:=
	\frac1h\left\{(E(\lambda+h)f,g)_{L^2(\Acal)}-(E(\lambda)f,g)_{L^2(\Acal)}\right\}\\
	&=
	\frac{1}{h}\int_{[0,1]} \sum_{ k \in \Bcal(\lambda , m ,h) } \left(P_m(\lambda_{m,k})f_m,g_m\right)_{L^2(\Sbb)}\,\di m \\
	&=
	\frac{1}{h}\int_{[0,1]} \sum_{ k \in \Bcal(\lambda, m ,h) } \,\widehat{f_m}(k) \overline{\widehat{g_m}(k)} \,\di m  .
	\end{align*}
We would now like to commute the integration and summation, to obtain the sum of integrals over small subintervals of $[0,1]$,
in order to then take the limit $h\to0$.
However, for fixed $h$  the sets $\Bcal(\lambda , m ,h)$ will contain arbitrarily many elements as $m\to0$.
Moreover, as $m\to0$ the preimages of $\lambda < ck m^{\alpha} \leq \lambda + h$ in $m\in[0,1]$ are no longer disjoint, resulting in many redundant integrations.
Indeed, intervals $\left(\frac{\lambda}{cm^{\alpha}} , \frac{\lambda+h}{cm^{\alpha}} \right]$ are of size larger than one if $m < (h/c)^{1/\alpha}$,
and then may contain more than one integer.
Our strategy is to split the previous integral into two parts
	\begin{equation*}
	\frac{1}{ h}\int_{[0,1]} \sum_{ k \in \Bcal(\lambda, m ,h) } \,\widehat{f_m}(k) \overline{\widehat{g_m}(k)} \,\di m
	=
	\mathcal{I}(h) + \mathcal{R}(h)
	\end{equation*}
where
	\begin{equation*}
	\mathcal{R}(h) = \frac{1}{h}\int_{[0 , (h/c)^{1/\alpha} )} \sum_{ k \in \Bcal(\lambda, m ,h) } \,\widehat{f_m}(k) \overline{\widehat{g_m}(k)} \,\di m
	\end{equation*}
is the part where the energy bands $\Bcal$ might contain more than one integer, and
	\begin{equation*}
	\mathcal{I}(h) = \frac{1}{ h}\int_{[(h/c)^{1/\alpha} , 1]} \sum_{ k \in \Bcal(\lambda, m ,h) } \,\widehat{f_m}(k) \overline{\widehat{g_m}(k)} \,\di m
	\end{equation*}
is the part where the energy bands can contain at most one integer. We now estimate these two integrals separately.

\emph{1. The term $\Rcal(h)$}. We  claim that  $ \lim_{h\to0} \mathcal{R}(h) = 0$.
First, observe that we can bound
	\begin{equation*}
	\left| \sum_{ k \in \Bcal(\lambda, m ,h) } \,\widehat{f_m}(k) \overline{\widehat{g_m}(k)} \right|
	\leq
	 \sum_{ k \in \Bcal(\lambda, m ,h) } \left|\widehat{f_m}(k) \overline{\widehat{g_m}(k)} \right|
	\leq
	\sum_{ k > \lambda / (cm^{\alpha}) } \, \left| \widehat{f_m}(k) \overline{\widehat{g_m}(k)} \right|.
	\end{equation*}
Taking $f,g \in \dot{\mathcal{H}}^{s,\gamma}_0$, we using  H\"older's inequality, and the inequality \eqref{ineq.Holder.0}, we obtain
	\begin{align*}
	\left| \mathcal{R}(h) \right|
	& \leq
	\frac{1}{  h}\int_{[0 , (h/c)^{1/\alpha} )} \sum_{ k > \lambda/(cm^{\alpha}) } \, |k|^{-2\gamma}
	\left( |k|^{\gamma}\left| \widehat{f_m}(k) \right| \right) \, \left( |k|^{\gamma} \left| \overline{\widehat{g_m}(k)} \right| \right) \,\di m \\
	& \leq
	\frac{1}{  h}\int_{[0 , (h/c)^{1/\alpha} )} \sup_{ k > \lambda/(cm^{\alpha}) } \left( |k|^{-2\gamma}  \right) \, \left\| f_m \right\|_{\dot{H}^{\gamma}} \left\| g_m \right\|_{\dot{H}^{\gamma}} \,\di m\\
	& \leq
	\frac{1}{  h}\int_{[0 , (h/c)^{1/\alpha} )} \left( \frac{\lambda}{cm^{\alpha}} \right)^{-2 \gamma} \left\| f_m \right\|_{\dot{H}^{\gamma}} \left\| g_m \right\|_{\dot{H}^{\gamma}} \,\di m \\
	& \leq
	C \, \|f\|_{s,\gamma} \, \|g\|_{s,\gamma} \, \frac{ \lambda^{-2 \gamma} }{   h } \int_{[0 , (h/c)^{1/\alpha} )} m^{2\alpha \gamma + 2(s-1/2)} \,\di m \\
	& =
	C \|f\|_{s,\gamma} \, \|g\|_{s,\gamma} \, { \lambda^{-2 \gamma} } h^{2 \gamma + 2s/\alpha - 1} .
	\end{align*}
As soon as the constraint \eqref{constraint} on $s$ and $\gamma$ is satisfied there holds $\lim_{ h \to 0} \mathcal{R}(h) = 0.$\\

\emph{2. The term $\Ical(h)$}. As the integration in this term is over  $m \in [(h/c)^{1/\alpha} , 1]$, the energy band $\Bcal(\lambda , m ,h)$ contains at most one integer.
For $\Bcal(\lambda , m ,h)$ to contain one integer, by definition $m$ has to be in an interval of the form $( \left( \frac{\lambda}{ck} \right)^{1/\alpha} , \left( \frac{\lambda+h}{ck} \right)^{1/\alpha} ]$ for some $k \in \N$.
As $m \in [(h/c)^{1/\alpha} , 1]$, this implies for the bounds of such an interval that
\[
	\left(\frac{h}{c}\right)^{1/\alpha}
	\leq
	\left( \frac{\lambda}{ck} \right)^{1/\alpha}
	\quad \text{and} \quad
	\left( \frac{\lambda+h}{ck} \right)^{1/\alpha}
	\leq
	1 .
\]
The integer $k\geq1$ has  to satisfy the bounds 
	\begin{equation*}
	k_0:=
	\max(1, \left\lfloor \lambda/c \right\rfloor)
	\leq
	k
	\leq
	 \lfloor \lambda/h \rfloor
	 =:
	 N(\lambda, h)
	\end{equation*}
Then we have
	\begin{align*}
	\mathcal{I}(h)
	& =
	\frac{1}{  h}\int_{[(h/c)^{1/\alpha} , 1]} \sum_{ k \in \Bcal(\lambda , m ,h) } \,\widehat{f_m}(k) \overline{\widehat{g_m}(k)} \,\di m \\
	& =
	\frac{1}{  h} \sum_{k_0\leq k\leq N(\lambda, h)}
	\int_{ ( \left( \frac{\lambda}{ck} \right)^{1/\alpha} , \left( \frac{\lambda+h}{ck}\right)^{1/\alpha} ] }
	\,\widehat{f_m}(k) \overline{\widehat{g_m}(k)} \,\di m .
	\end{align*}

Next, we use the same kind of computation as in the proof of Proposition \ref{thm:toy}, where the prefactor $\frac{1}{\Ecal'(m)}$ appears. In the current computation, the function $\Ecal(m)$ becomes $k\phi(m)=ckm^\alpha$ so that $\Ecal'(m)$ becomes $ck\alpha m^{\alpha-1}$. Substituting $\left(\frac{\lambda}{ck}\right)^{1/\alpha}$ for $m$  we find that the prefactor should be $\frac{1}{ck\alpha} \left( \frac{\lambda}{ck} \right)^{1/\alpha -1}$.
Furthermore, we use Lemma \ref{prop.Holder.Fourier} which ensures some uniform regularity for the Fourier coefficients and allows us to make pointwise (in $m$) evaluations.
This leads to
	\begin{equation*}
	\lim_{h\to0} \frac{1}{h} \int_{ ( \left( \frac{\lambda}{ck} \right)^{1/\alpha} , \left( \frac{\lambda+h}{ck} \right)^{1/\alpha} ] }
	\,\widehat{f_m}(k) \overline{\widehat{g_m}(k)} \,\di m
	 =
	\frac{1}{ck\alpha} \left( \frac{\lambda}{ck} \right)^{1/\alpha -1} \,\widehat{f}_{(\frac{\lambda}{ck})^{1/\alpha}}(k) \overline{\widehat{g}_{(\frac{\lambda}{ck})^{1/\alpha}}(k)}.
	\end{equation*}
Using the fact that $\lim_{h\to0}N(\lambda,h)=+\infty$, we get
	\begin{equation*}
	\lim_{h\to0} \mathcal{I}(h)
	=
	\sum_{k \geq k_0 }
	\frac{1}{c\alpha k} \left( \frac{\lambda}{ck} \right)^{1/\alpha -1} \,\widehat{f}_{(\frac{\lambda}{ck})^{1/\alpha}}(k) \, \overline{\widehat{g}_{(\frac{\lambda}{ck})^{1/\alpha}}(k)}
	\end{equation*}
by uniform boundedness.
Using the inequality \eqref{ineq.Holder.Fourier} and the fact that $f_0=0$, we have
%
	\begin{align*}
	\left| \lim_{h\to0} \mathcal{I}(h) \right|
	&
	= \frac{1}{c\alpha} \left|
	\sum_{k \geq k_0 }  \frac{1}{k} \left( \frac{\lambda}{ck} \right)^{1/\alpha -1}
	\widehat{f}_{(\frac{\lambda}{ck})^{1/\alpha}}(k) \, \overline{\widehat{g}_{(\frac{\lambda}{ck})^{1/\alpha}}(k)}
	\right| \\
	& \leq
	C \, \|f\|_{s,\gamma}  \|g\|_{s,\gamma}  { \lambda^{1/\alpha -1} } \sum_{k \geq k_0 } k^{{}-1/\alpha -2\gamma} \left( \frac{\lambda}{ck} \right)^{2(s-1/2)/\alpha} \\
	& \leq
	C \, \|f\|_{s,\gamma} \|g\|_{s,\gamma} { \lambda^{2s/\alpha -1} } \sum_{k \geq k_0 } k^{-(1/\alpha +2\gamma + 2(s-1/2)/\alpha)} \\
	& \leq
	C \, \|f\|_{s,\gamma} \|g\|_{s,\gamma} { \lambda^{2s/\alpha -1} } \sum_{k \geq 1 } k^{{}-(2\gamma + 2s/\alpha)}
	\end{align*}
 where we use the fact that $ k_0 \geq 1$ in the last inequality.
The series is finite if the constraint \eqref{constraint} on $s$ and $\gamma$ is satisfied.
This completes the proof.
\end{proof}

To prove Theorem \ref{thm:main} we recall the following result from \cite{Ben-Artzi2019c}:
\begin{proposition}[Corollary 1.8 in \cite{Ben-Artzi2019c}]\label{thm:abstract}
Let $A:D(A)\subset\Hcal\to\Hcal$ be self-adjoint and assume that there exist  a Banach subspace $\Xcal\subset\Hcal$ which is dense in $\Hcal$  in the topology of $\Hcal$ that is continuously embedded in $\Hcal$ (and therefore the norm $\|\cdot\|_\Xcal$ is stronger than the norm $\|\cdot\|_\Hcal$), and positive numbers $C,p,r>0$  such that:
	\begin{equation*}
	\left|\frac{\di}{\di\lambda}\left(E(\lambda)f,g\right)_{\Hcal}\right|
	\leq
	C |\lambda|^{p-1}\|f\|_\Xcal\|g\|_\Xcal,\qquad\forall f,g\in\Xcal,\,\forall\lambda\in (-r,r)\setminus\{0\}.
	\end{equation*}
	Then  $P^Tf:=\frac{1}{2T}\int_{-T}^Te^{itA}f\,\di t$ converges to the orthogonal projection of $\Hcal$ onto $\ker A$ (denoted $P$) uniformly:
	\begin{equation*}\label{eq:conv-rate-frac}
	\|P^T-P\|_{\Xcal\to\Hcal}\leq C(p)T^{-\ell/2},\qquad\forall T>1,
	\end{equation*}
where $\ell=\min\{p-\varepsilon,2\}$ for any $\varepsilon>0$.
\end{proposition}

This result readily allows us to prove our main theorem:

\begin{proof}[Proof of Theorem \ref{thm:main}]
From Proposition \ref{thm:main2} we know that  $\frac{\di}{\di\lambda}E(\lambda)$ has the bound $|\lambda|^{\frac{2s}{\alpha}-1}$ in the subspace $\dot{\mathcal{H}}^{s,\gamma}_0$ in a punctured neighborhood of $\lambda=0$. From Proposition \ref{thm:abstract} we know that this implies that
	\begin{equation*}\label{eq:conv-rate-frac}
	\|P^T\|_{\dot{\mathcal{H}}^{s,\gamma}_0\to L^2(\Acal)}=\|P^T-P\|_{\dot{\mathcal{H}}^{s,\gamma}_0\to L^2(\Acal)}\leq CT^{-\ell/2},\qquad\forall T>1,
	\end{equation*}
where $\ell=\min\{\frac{2s}{\alpha}-\varepsilon,2\}$ for any $\varepsilon>0$.
This completes the proof.
\end{proof}

\bibliography{library}

\begin{thebibliography}{10}

\bibitem{Agmon1975}
S.~Agmon.
\newblock {Spectral properties of Schr{\"{o}}dinger operators and scattering
  theory}.
\newblock {\em Ann. Scuola Norm. Sup. Pisa Cl. Sci.}, 2(2):151--218, 1975.

\bibitem{Ben-Artzi2013d}
J.~Ben-Artzi.
\newblock {On the spectrum of shear flows and uniform ergodic theorems}.
\newblock {\em Journal of Functional Analysis}, 267(1):299--322, jul 2014.

\bibitem{Ben-Artzi2014}
J.~Ben-Artzi.
\newblock {Instabilities in kinetic theory and their relationship to the
  ergodic theorem}.
\newblock {\em Contemporary Mathematics}, 653:15, dec 2015.

\bibitem{Ben-Artzi2019c}
J.~Ben-Artzi and B.~Morisse.
\newblock {Uniform convergence in von Neumann's ergodic theorem in the absence
  of a spectral gap}.
\newblock {\em Ergodic Theory and Dynamical Systems}, 41(6):1601--1611, jun
  2021.

\bibitem{Cox2013a}
G.~Cox.
\newblock {The L 2 Essential Spectrum of the 2D Euler Operator}.
\newblock {\em Journal of Mathematical Fluid Mechanics}, 16(3):419--429, sep
  2014.

\bibitem{Kachurovskii2023a}
A.~G. Kachurovskii, I.~V. Podvigin, and A.~Z. Khakimbaev.
\newblock {Uniform Convergence on Subspaces in the von Neumann Ergodic Theorem
  with Discrete Time}.
\newblock {\em Mathematical Notes}, 113(5-6):680--693, 2023.

\bibitem{Kachurovskii2023}
A.~G. Kachurovskii, I.~V. Podvigin, and V.~E. Todikov.
\newblock {Uniform Convergence on Subspaces in Von Neumann'S Ergodic Theorem
  With Continuous Time}.
\newblock {\em Siberian Electronic Mathematical Reports}, 20(1):183--206, 2023.

\bibitem{Kato1995}
T.~Kato.
\newblock {\em {Perturbation Theory for Linear Operators}}.
\newblock Springer-Verlag, 1995.

\bibitem{Lions1972}
J.~L. Lions and E.~Magenes.
\newblock {\em {Non-homogeneous boundary value problems and applications,
  Volume 1}}.
\newblock Springer-Verlag, 1972.

\bibitem{Moser2005}
J.~Moser and E.~Zehnder.
\newblock {\em {Notes on Dynamical Systems}}.
\newblock American Mathematical Society, 2005.

\bibitem{Reed1978}
M.~Reed and B.~Simon.
\newblock {\em {Methods of modern mathematical physics volume 4: Analysis of
  operators}}.
\newblock 1978.

\bibitem{Simon1990}
J.~Simon.
\newblock {Sobolev, Besov and Nikolskii fractional spaces: Imbeddings and
  comparisons for vector valued spaces on an interval}.
\newblock {\em Annali di Matematica Pura ed Applicata}, 157(1):117--148, dec
  1990.

\bibitem{VonNeumann1932a}
J.~von Neumann.
\newblock {Proof of the quasi-ergodic hypothesis}.
\newblock {\em Proceedings of the National Academy of Sciences}, 18(2):70--82,
  1932.

\bibitem{Yao2014}
Y.~Yao and A.~Zlato{\v{s}}.
\newblock {Mixing and un-mixing by incompressible flows}.
\newblock {\em Journal of the European Mathematical Society}, 19(7):1911--1948,
  2017.

\end{thebibliography}
\bibliographystyle{abbrv}
\end{document}